\documentclass[12pt,reqno]{amsart}

\newcommand\version{November 4, 2017}

\usepackage{amsmath,amsfonts,amsthm,amssymb,amsxtra}

\newtheorem{theorem}{Theorem}[section]

\newtheorem{lemma}[theorem]{Lemma}
\newtheorem{corollary}[theorem]{Corollary}

\theoremstyle{definition}

\theoremstyle{remark}

\numberwithin{equation}{section}

\newcommand{\cl}{\mathrm{cl}}

\renewcommand{\epsilon}{\varepsilon}

\newcommand{\N}{\mathbb{N}}

\renewcommand{\phi}{\varphi}
\newcommand{\R}{\mathbb{R}}

\newcommand{\Sph}{\mathbb{S}}

\DeclareMathOperator{\dist}{dist}

\DeclareMathOperator{\ran}{ran}
\DeclareMathOperator{\re}{Re}
\DeclareMathOperator{\spec}{spec}
\DeclareMathOperator{\supp}{supp}

\DeclareMathOperator{\tr}{Tr}

\begin{document}

\title[Eigenvalue bounds for the fractional Laplacian --- \version]{Eigenvalue bounds for the fractional Laplacian:\\ A review}

\author{Rupert L. Frank}
\address{Rupert L. Frank, Ludwig-Maximilans Universit\"at M\"unchen, Theresienstr. 39, 80333 M\"unchen, Germany, and Mathematics 253-37, Caltech, Pasadena, CA 91125, USA}
\email{rlfrank@caltech.edu}

\begin{abstract}
We review some recent results on eigenvalues of fractional Laplacians and fractional Schr\"odinger operators. We discuss, in particular, Lieb--Thirring inequalities and their generalizations, as well as semi-classical asymptotics.
\end{abstract}


\maketitle


\makeatletter{\renewcommand*{\@makefnmark}{}
\footnotetext{\copyright\, 2017 by the author. This paper may be reproduced, in its entirety, for non-commercial purposes.}\makeatother}


\section{Introduction} 

An attempt is made, at the request of the editors of this volume to whom the author is grateful, to review some recent developments concerning eigenvalues of fractional Laplacians and fractional Schr\"odinger operators. Such review is necessarily incomplete and biased towards the author's interests. It is hoped, however, that this collection of results will provide a useful snapshot of a certain line of research and that the open questions mentioned here stimulate some further research.

As is well known, the fractional Laplacian appears in many different areas in connection with non-local phenomena. Here we are particularly interested in problems related to quantum mechanics, where the square root of the Laplacian is used to model relativistic effects. Early works on the one-body and many-body theory include \cite{He,CaMaSi} and \cite{Da,LiTh2,LiYa1,FeLl,LiYa2}, respectively, and we refer to these for further physical motivations.

Let us define the operators in question. For an open set $\Omega\subset\R^d$ we denote by $\mathring H^s(\Omega)$ the set of all functions in the Sobolev space $H^s(\R^d)$ which vanish almost everywhere in $\R^d\setminus\Omega$. We denote the Fourier transform of $\psi$ by
$$
\hat\psi(p) := \frac{1}{(2\pi)^{d/2}} \int_{\R^d} e^{-ip\cdot x} \psi(x)\,dx \,.
$$
The non-negative quadratic form
$$
\int_{\R^d} |p|^{2s} |\hat\psi(p)|^2\,dp \,,
\qquad \psi\in \mathring H^s(\Omega) \,,
$$
(note that $\psi$ is zero almost everywhere on $\R^d\setminus\Omega$) is closed in the Hilbert space $L^2(\Omega)$ and therefore generates a self-adjoint, non-negative operator
$$
H^{(s)}_\Omega
\qquad\text{in}\ L^2(\Omega) \,.
$$
For $0<s<1$ we call $H^{(s)}_\Omega$ the \emph{fractional Laplacian} in $\Omega$. When $s=1$, this construction gives the usual Dirichlet Laplacian, which we denote by $-\Delta_\Omega=H_\Omega^{(1)}$. When $\Omega=\R^d$, then $H_{\R^d}^{(s)}$ coincides with the fractional power $s$ (in the sense of the functional calculus) of the operator $-\Delta_{\R^d}$ and we will simplify notation by writing $(-\Delta)^s=H^{(s)}_{\R^d}$. It is important to note that, if $\Omega\neq\R^d$ (up to sets of capacity zero), then $H^{(s)}_\Omega$ does \emph{not} coincide with the fractional power of the operator $-\Delta_\Omega$ and, in fact, the comparison of these two operators is one of the recurring themes in this review.

There is a useful alternative expression for the fractional Laplacian, namely,
$$
\int_{\R^d} |p|^{2s} |\hat\psi(p)|^2\,dp = a_{d,s} \iint_{\R^d\times\R^d} \frac{|\psi(x)-\psi(y)|^2}{|x-y|^{d+2s}} \,dx\,dy
$$
for all $\psi\in H^s(\R^d)$ with
\begin{equation}
\label{eq:formconst}
a_{d,s} = 2^{2s-1} \pi^{-d/2} \frac{\Gamma(\frac{d+2s}{2})}{|\Gamma(-s)|} \,.
\end{equation}
This is a classical computation, which we recall in Appendix \ref{app:const}.

Besides the fractional Laplacian on an open set we will also be interested in the fractional Schr\"odinger operator $(-\Delta)^s + V$. Heuristically, the connection between the two operators is that the fractional Laplacian in $\Omega$ is the special case of the fractional Schr\"odinger operator with the potential $V$ which equals $0$ on $\Omega$ and $+\infty$ on its complement. This intuition can be made precise as a limiting theorem, at least in the case of a not too irregular boundary, but we will not make use of this here. Nevertheless, it is useful to keep this connection in mind when comparing the results for both operators.

As we said, our main concern here are eigenvalue bounds for $H^{(s)}_\Omega$ and $(-\Delta)^s+V$. It is technically convenient to consider, instead of eigenvalues, the numbers given by the variational principle. Namely, for a general self-adjoint operator $A$ with quadratic form $a$ in a Hilbert space and for $n\in\N$ we define
$$
E_n(A) := \sup_{\psi_1,\ldots,\psi_{n-1}} \left( \inf_{0\neq \psi\bot\psi_1,\ldots,\psi_{n-1}} \frac{a[\psi]}{\|\psi\|^2} \right).
$$
According to the variational principle (see, e.g., \cite[Theorem XIII.1]{ReSi}), if $E_n(A)<\inf\text{ess-}\spec(A)$, then $E_n(A)$ is the $n$-th eigenvalue of $A$, counting multiplicities. In general, however, $E_n(A)$ need not be an eigenvalue. Since our tools in this paper are of variational nature, they lead naturally to inequalities for $E_n(A)$, independently of whether or not it actually is an eigenvalue.

\medskip

Let us briefly outline the structure this review. In Section \ref{sec:single} we begin with lower bounds on the ground state energies $E_1(H^{(s)}_\Omega)$ and $E_1((-\Delta)^s+V)$. These lower bounds come naturally from the shape optimization problems of minimizing $E_1(H^{(s)}_\Omega)$ among all $\Omega$ with given measure and minimizing $E_1((-\Delta)^s+V)$ among all $V$ with given $L^p$ norm. The (classical) answers are given in Theorems \ref{fk} and \ref{ke}. We then turn to comparing the eigenvalues of the operators $H^{(s)}_\Omega$ and $(-\Delta_\Omega)^s$ and recall a theorem from \cite{ChSo}.

In Section \ref{sec:asymp} we discuss the asymptotics of $E_n(H_\Omega^{(s)})$ as $n\to\infty$ and of $\#\{ n: E_n((-\Delta)^s+\alpha V)<0\}$ as $\alpha\to\infty$. Both questions are closely related, because studying the asymptotics of $E_n(H_\Omega^{(s)})$ as $n\to\infty$ is the same as studying $\#\{ n: E_n(H_\Omega^{(s)})<\mu\}$ as $\mu\to\infty$, which is the same as studying $\#\{ n: E_n(h^{2s} H_\Omega^{(s)} -1)<0\}$ as $h\to 0$. Clearly, studying $\#\{ n: E_n((-\Delta)^s+\alpha V)<0\}$ as $\alpha\to\infty$ is the same as studying $\#\{ n: E_n(h^{2s} (-\Delta)^s+ V)<0\}$ as $h\to 0$, so both questions correspond to a semi-classical limit with an effective Planck constant $h$ tending to zero. While the leading term in the asymptotics is well known and given by a Weyl-type formula, there are still open questions corresponding to subleading corrections.

In Section \ref{sec:unif} we supplement the asymptotic results on the number and sums of eigenvalues by `uniform' inequalities which hold not only in the asymptotic regimes considered in the previous section. The important feature of these inequalities is, however, that they have a form reminiscent of the asymptotics. We present such eigenvalue bounds not only for $H_\Omega^{(s)}$ and $(-\Delta)^s +V$, but also for operators of the form $(-\Delta)^s -W +V$, where $W$ is an explicit `Hardy weight'.

We conclude with a short Section \ref{sec:omitted} on (some of) the topics that we have not treated in this paper.

\subsection*{Acknowledgement}
The author would like to thank B. Dyda, L. Geisinger, E. Lenzmann, E. Lieb, R. Seiringer and L. Silvestre for collaborations involving the fractional Laplacian. He thanks R. Ba\~nuelos, M. Kwa\'snicki, F. Maggi and R. Song for helpful comments. Partial support by U.S. National Science Foundation DMS-1363432 is acknowledged.


\section{Bounds on single eigenvalues}\label{sec:single}

\subsection{The fractional Faber--Krahn inequality}

We recall that $E_1(H_\Omega^{(s)})$ denotes the ground state energy of the fractional Laplacian on an open set $\Omega\subset\R^d$. Using Sobolev interpolation inequalities on $\R^d$ (see, for instance, \eqref{eq:sobolevinterpol} below) and H\"older's inequality it is easy to prove that
$$
E_1(H_\Omega^{(s)}) \geq C_{d,s} |\Omega|^{-2s/d}
$$
for some positive constant $C_{d,s}$ depending only on $d$ and $s$. The fractional Faber--Krahn inequality in the following theorem says that the optimal value of the constant $C_{d,s}$ is attained when $\Omega$ is a ball. We recall that for any measurable set $E\subset\R^d$ of finite measure, $E^*$ denotes the centered, open ball with radius determined such that $|E^*|=|E|$.

\begin{theorem}\label{fk}
Let $\Omega\subset\R^d$ be open with finite measure. Then
$$
E_1(H_\Omega^{(s)}) \geq E_1(H_{\Omega^*}^{(s)})
$$
with equality if and only if $\Omega$ is a ball.
\end{theorem}

This theorem follows easily using symmetric decreasing rearrangement (see, e.g., \cite{LiLo} for a textbook presentation). We know \cite{AL} that
\begin{equation}
\label{eq:rearr}
\|(-\Delta)^{s/2}\psi\|^2 \geq \|(-\Delta)^{s/2}\psi^*\|^2 \,, 
\end{equation}
where $\psi^*$ denotes the symmetric decreasing rearrangement of $\psi$, and since $\|\psi^*\|^2= \|\psi\|^2$ and $\psi^*$ is supported in $\overline{\Omega^*}$, we obtain the inequality in the theorem. The uniqueness of the ball follows from the strictness statement for \eqref{eq:rearr}, see \cite{BuHi,FrSe}. (We also mention that a version of \eqref{eq:rearr} for functions on a interval appears in \cite{GaRo}.) An alternative proof of Theorem~\ref{fk}, based on a comparison result for the corresponding heat equations, can be found in \cite{SiVaVo}. For results related to and generalizing Theorem~\ref{fk}, see \cite{BaMH}. 

It would be interesting to supplement Theorem \ref{fk} with a stability result analogous to \cite{FuMaPr,BrDPVe}, namely to show that $E_1(H_\Omega^{(s)}) - E_1(H_{\Omega^*}^{(s)})$ is bounded from below by a constant (depending only on $s$ and $d$) times $|\Omega|^{-2s/d-2} \inf\{ |B\Delta\Omega|^2:\ B \ \text{ball with}\ |B|=|\Omega|\}$.

Theorem \ref{fk} corresponds to minimizing $E_1(H_\Omega^{(s)})$ among all sets $\Omega$ with given measure. Another interesting problem is to minimize $E_1(H_\Omega^{(s)})$ among all \emph{convex} sets $\Omega$ with given inner radius $r_\text{in}(\Omega):=\sup_{x\in\Omega} \dist(x,\Omega^c)$. Optimal results for this question appear in \cite{BLM,Me}.


\subsection{The fractional Keller inequality}

We recall that $E_1((-\Delta)^s +V)$ denotes the ground state energy of the fractional Schr\"odinger operator. In \cite{Ke} Keller asked for $s=1$ how small the ground state energy can be for a given $L^p$ norm of the potential; see also \cite{LiTh}. The following theorem generalizes this result to the fractional case.

\begin{theorem}
\label{ke}
Let $d\geq1$, $0<s<1$ and $\gamma>0$. If $d=1$ and $s>1/2$ we assume in addition that $\gamma\geq 1-1/(2s)$. Then
$$
K_{\gamma,d,s} := - \inf_V \frac{E_1((-\Delta)^s+V)}{\|V\|_{\gamma+d/(2s)}^{1+ d/(2s\gamma)}} <\infty \,.
$$
Moreover, there is a positive, radial, symmetric decreasing function $W$ such that the inequality
\begin{equation}
\label{eq:ke}
E_1((-\Delta)^s+V) \geq - K_{\gamma,d,s} \left( \int_{\R^d} |V|^{\gamma+d/(2s)}\,dx \right)^{1/\gamma}
\end{equation}
is strict unless $V= -b^{-2s} W((x-a)/b)$ for some $a\in\R^d$ and $b>0$.
\end{theorem}

Let us briefly sketch the proof. The key idea (essentially contained in \cite{LiTh} for $s=1$) is that the inequality $K_{\gamma,d,s}<\infty$ is equivalent to a Sobolev interpolation inequality. According to the variational definition of $E_1((-\Delta)^s +V)$ we have
$$
- K_{\gamma,d,s} = \inf_V \inf_\psi \frac{\|(-\Delta)^{s/2}\psi\|^2 + \int_{\R^d} V|\psi|^2\,dx}{\|\psi\|^2 \ \|V\|_{\gamma+d/(2s)}^{1+ d/(2s\gamma)}}
$$
Since the quotient in this formula remains invariant if we replace both $V(x)$ by $b^2 V(bx)$ and $\psi(x)$ by $c\psi(bx)$ for arbitrary $b,c>0$, we can restrict the infimum to potentials $V$ with $\|V\|_{\gamma+d/(2s)}=1$ and to functions $\psi$ with $\|\psi\|=1$. Moreover, since the quotient does not increase if we replace $V$ by $-|V|$ we can restrict the infimum to potentials $V\leq 0$. We summarize these findings as
$$
- K_{\gamma,d,s} = \inf \left\{ \mathcal E_q[\psi] + \mathcal H_q[\psi,U] :\ U\geq 0\,,\ \|\psi\|=\|U\|_{q/(q-2)}=1 \right\}
$$
with $q\geq 2$ such that $1/(\gamma+d/(2s)) + 2/q = 1$,
$$
\mathcal E_q[\psi] = \|(-\Delta)^{s/2}\psi\|^2 - \|\psi\|_q^2
$$
and
$$
\mathcal H_q[\psi,U] = \|\psi\|_q^2 - \int_{\R^d} U|\psi|^2\,dx \,.
$$
By H\"older's inequality we have $\mathcal H_q[\psi,U]\geq 0$ for $U\geq 0$ with $\|U\|_{q/(q-2)}=1$, and equality holds if and only if $U = (|\psi|/\|\psi\|_q)^{q-2}$. Thus, $K_{\gamma,d,s}<\infty$ is equivalent to
\begin{equation}
\label{eq:keeinf}
\inf\left\{ \mathcal E_q[\psi] :\ \|\psi\|=1 \right\} > - \infty \,,
\end{equation}
and there is a bijective correspondence between $V$'s realizing equality in \eqref{eq:ke} and $\psi$'s realizing the infimum in \eqref{eq:keeinf}. The statement \eqref{eq:keeinf} is, by scaling, equivalent to the Sobolev interpolation inequality
\begin{equation}
\label{eq:sobolevinterpol}
\| (-\Delta)^{s/2}\psi\|^{2\theta} \|\psi\|^{2(1-\theta)} \geq \mathcal S_{d,q,s} \|\psi\|_q^2
\end{equation}
with a constant $\mathcal S_{d,q,s}>0$ (and some $\theta\in(0,1)$ uniquely determined by scaling). This inequality is well known to hold for $2\leq q\leq 2d/(d-2s)$ if $d>2s$, for $2\leq q<\infty$ if $d=2s$ and for $2\leq q\leq\infty$ if $d<2s$. Therefore, we deduce that $K_{\gamma,d,s}<\infty$ under the assumptions on $\gamma$ in the theorem. Moreover, if $\mathcal S_{d,q,s}$ denotes the optimal constant in \eqref{eq:sobolevinterpol}, then it is also well-known that there is a minimizer $\psi$ for which equality holds (see, for instance, \cite{CaFrLi} for a proof for $s=1$; the necessary modifications for $s<1$ are, for instance, in \cite{BeFrVi}). By the rearrangement inequality \eqref{eq:rearr}, this $\psi$ can be chosen positive, radial and symmetric decreasing. It was recently proved in \cite{FrLe,FrLeSi} that there is a unique function $Q$ such that any function achieving equality in \eqref{eq:sobolevinterpol} coincides with $Q$ after translation, dilation and multiplication by a constant, which leads to the uniqueness statement in Theorem \ref{ke}. This completes our sketch of the proof of the theorem.

We expect that the method from \cite{CaFrLi}, together with the non-degeneracy results from \cite{FrLe,FrLeSi}, leads to a stability version of \eqref{eq:ke}.


\subsection{Comparing eigenvalues of $H_\Omega^{(s)}$ and $(-\Delta_\Omega)^s$}

It is important to distinguish between $H_\Omega^{(s)}$, the fractional Laplacian on $\Omega$, and the fractional power $(-\Delta_\Omega)^s$ of the Dirichlet Laplacian. These two operators are different, but, as shown in the following theorem, the first one is always less or equal than the second one. We recall that for two operators $A,B$, which are bounded from below, we write $A\leq B$ if their quadratic forms $a,b$ with form domains $\mathcal D[a]$, $\mathcal D[b]$ satisfy $\mathcal D[a]\supset\mathcal D[b]$ and $a[u]\leq b[u]$ for every $u\in\mathcal D[b]$. Note that $A\leq B$ implies $E_n(A)\leq E_n(B)$ for all $n\in\N$.

\begin{theorem}\label{evcomp}
Let $\Omega\subset\R^d$ be open and $0<s<1$. Then
\begin{equation}
\label{eq:opcomp}
H_\Omega^{(s)} \leq (-\Delta_\Omega)^s
\end{equation}
In particular, 
\begin{equation}
\label{eq:evcomp}
E_n(H_\Omega^{(s)}) \leq E_n((-\Delta_\Omega)^s) = \left( E_n(-\Delta_\Omega) \right)^s 
\qquad\text{for all}\ n\in\N \,.
\end{equation}
Moreover, unless $\R^d\setminus\Omega$ has zero capacity, the operators $H_\Omega^{(s)}$ and $(-\Delta_\Omega)^s$ do not coincide.
\end{theorem}

The first part of the theorem is due to Chen--Song \cite{ChSo} (see also \cite{deBlMH} and its generalization in \cite{ChSo2}), which extends earlier results in \cite{BaKu04} for $s=1/2$ and in \cite{deBl} for $s$ irrational. The second part concerning strictness is from \cite{FrGe}, where it is also shown that, in a certain sense, $E_n(H_\Omega^{(s)})$ and $E_n((-\Delta_\Omega)^s)$ have the same leading term as $n\to\infty$, but a different subleading term; see Corollary \ref{getoursumsecondcor} below for a precise statement. An alternative proof of Theorem \ref{evcomp}, which yields strict inequality in \eqref{eq:evcomp} for any $n$ for bounded $\Omega$, is in \cite{MuNa} and is based on the Caffarelli--Silvestre extension technique \cite{CaSi}. In fact, it was recently shown in \cite{KwLaSi}, using Jensen's inequality, that $(0,1)\ni s\mapsto E_n(H_\Omega^{(s)})^{1/s}$ is \emph{strictly} increasing for any $n$ if $\Omega$ is bounded.

Let us sketch the idea of the proof of Theorem \ref{evcomp} in \cite{FrGe}. It is based on the observation that, if $A$ is a non-negative operator in a Hilbert space with trivial kernel, $P$ an orthogonal projection and $\phi$ an operator monotone function on $(0,\infty)$, then
\begin{equation}
\label{eq:opmonineq}
P\phi(PAP)P \geq P\phi(A)P \,.
\end{equation}
(This is closely related to the Sherman--Davis inequality, see, e.g., \cite[Theorem 4.19]{Ca}.) Using Loewner's integral representation of operator monotone functions (see, for instance, \cite{Bat} or \cite[Theorem 1.6]{Si3}), \eqref{eq:opmonineq} follows from
\begin{equation}
\label{eq:opmonineq1}
P(PAP)^{-1} P \leq P A^{-1} P \,,
\end{equation}
which, in turn, can be proved using a variational characterization of the inverse operator in the spirit of Dirichlet's principle. Inequality \eqref{eq:opcomp} follows immediately from \eqref{eq:opmonineq} with the choice $A=-\Delta$ in $L^2(\R^d)$, $P$ = multiplication by the characteristic function of $\Omega$ (note that $H^{(s)}_\Omega = PA^sP$ in the quadratic form sense) and $\phi(E)=E^s$ (which is operator monotone for $0<s<1$). Note that this argument gives an analogue of \eqref{eq:opcomp} for any operator monotone function. If we only want \eqref{eq:opcomp}, we do not need Loewner's theorem, but only the integral representation
$$
E^s = \frac{\sin(\pi s)}{\pi} \int_0^\infty t^s \left( \frac{1}{t} - \frac{1}{t+E} \right)dt
\qquad\text{if}\ 0<s<1 \,.
$$

Analyzing the cases of equality in \eqref{eq:opmonineq1} shows that, under the assumption that $\phi$ is not affine linear, equality in \eqref{eq:opmonineq} holds iff $\ran P$ is a reducing subspace of $A$. (This is stated in \cite{FrGe} only for positive definite $A$, which is needed for \eqref{eq:opmonineq1}, but when passing from \eqref{eq:opmonineq1} to \eqref{eq:opmonineq} one always has a positive definite operator.) Since $L^2(\Omega)$ is not a reducing subspace for $-\Delta$ in $L^2(\R^d)$ unless $\R^d\setminus\Omega$ has capacity zero, we obtain the second part of the theorem.

\medskip

While Theorem \ref{evcomp} gives an upper bound on $E_n(H_\Omega^{(s)})$ in terms of $E_n(-\Delta_\Omega^s)$, the following theorem, also due to Chen--Song \cite{ChSo}, yields a lower bound.

\begin{theorem}\label{evcomplower}
Let $\Omega\subset\R^d$ be bounded and satisfy the exterior cone condition and let $0<s<1$. Then there is a $c_{\Omega,s}>0$ such that
\begin{equation}
\label{eq:evcomplower}
E_n(H_\Omega^{(s)}) \geq c_{\Omega,s}\, E_n((-\Delta_\Omega)^s) = c_{\Omega,s} \left( E_n(-\Delta_\Omega) \right)^s 
\qquad\text{for all}\ n\in\N \,.
\end{equation}
If $\Omega$ is convex, \eqref{eq:evcomplower} holds with $c_{\Omega,s}=1/2$.
\end{theorem}

We note that Theorem \ref{evcomplower} allows one to obtain lower bounds on $E_n(H_\Omega^{(s)})$ from lower bounds on $E_n(-\Delta_\Omega)$. For instance, one can show that for convex domains $E_1(H_\Omega^{(s)})$ is bounded from below by a constant times $r_{\text{in}}(\Omega)^{-2s}$ \cite{ChSo}. This gives weaker inequalities, however, than the direct approach in \cite{BLM,Me}.


\section{Eigenvalue asymptotics}

\subsection{Eigenvalue asymptotics for the fractional Laplacian}\label{sec:asymp}

From a (fractional analogue) of Rellich's compactness lemma we know that $H_\Omega^{(s)}$ has purely discrete spectrum when $\Omega\subset\R^d$ has finite measure. In this subsection we discuss the asymptotics of the eigenvalues $E_n(H_\Omega^{(s)})$ as $n\to\infty$. The basic result is due to Blumenthal and Getour \cite{BlGe} (see also \cite[Rem. 2.2]{BaKu08} and \cite{Ge}).

\begin{theorem}\label{getour}
Let $\Omega\subset\R^d$ be open with finite measure. Then
\begin{equation}
\label{eq:getour}
\lim_{n\to\infty} \frac{E_n(H_\Omega^{(s)})}{n^{2s/d}} = (2\pi)^{2s} \omega_d^{-2s/d} |\Omega|^{-2s/d}
\end{equation}
with $\omega_d = |\{\xi\in\R^d:\ |\xi|<1\}|$.
\end{theorem}

Alternatively, one can write \eqref{eq:getour} as
\begin{equation}
\label{eq:getouralt}
\lim_{\mu\to\infty} \mu^{-d/(2s)} N(\mu,H_\Omega^{(s)}) = (2\pi)^{-d} \omega_d |\Omega| \,,
\end{equation}
where for an arbitrary self-adjoint operator $A$, which is bounded from below, we set $N(\mu,A)=\#\{ n:\ E_n(A)<\mu\}$. If $A$ has discrete spectrum in $(-\infty,\mu)$, then $N(\mu,A)$ denotes the total number of eigenvalues below $\mu$, counting multiplicities.

For later purposes we record that \eqref{eq:getour} implies
\begin{equation}
\label{eq:getoursum}
\lim_{N\to\infty} N^{-1-2s/d} \sum_{n=1}^N E_n(H_\Omega^{(s)}) = \frac{d}{d+2s} (2\pi)^{2s} \omega_d^{-2s/d} |\Omega|^{-2s/d} \,.
\end{equation}
We also note that \eqref{eq:getouralt} and integration in $\mu$ shows that, for any $\gamma>0$,
\begin{equation}
\label{eq:getourgamma}
\lim_{\mu\to\infty} \mu^{-\gamma-d/(2s)} \tr\left(H_\Omega^{(s)}-\mu\right)_-^\gamma = L_{\gamma,d,s}^\cl |\Omega| \,,
\end{equation}
where
$$
\tr (H_\Omega^{(s)}-\mu)_-^\gamma = \sum_n \left( E_n(H_\Omega^{(s)}) - \mu \right)_-^\gamma = \gamma \int_0^\infty N(\mu,H_\Omega^{(s)}) \mu^{\gamma-1}\,d\mu
$$
and
\begin{equation}
\label{eq:weylconst}
L_{\gamma,d,s}^\cl := \frac{1}{(2\pi)^d} \int_{\R^d} \left( |p|^
{2s}-1 \right)_-^\gamma dp = \frac{\omega_d}{(2\pi)^d} \frac{\Gamma(\gamma+1)\,\Gamma(\frac{d}{2s}+1)}{\Gamma(\gamma+\frac{d}{2s}+1)} \,.
\end{equation}

\medskip

A classical result of Weyl states that
$$
\lim_{\mu\to\infty} \mu^{-d/2} N(\mu,-\Delta_\Omega) = (2\pi)^{-d} \omega_d \,|\Omega| \,,
$$
and therefore, by the spectral theorem,
$$
\lim_{\mu\to\infty} \mu^{-d/(2s)} N(\mu,(-\Delta_\Omega)^s) = \lim_{\mu'\to\infty} (\mu')^{-d/2} N(\mu',-\Delta_\Omega) = (2\pi)^{-d} \omega_d\, |\Omega| \,.
$$
Comparing this with \eqref{eq:getour} we see that $E_n(H_\Omega^{(s)})$ and $E_n(-\Delta_\Omega^s) = \left( E_n(-\Delta_\Omega)\right)^s$ coincide to leading order as $n\to\infty$. In the following we will be interested in subleading corrections to the asymptotics in Theorem \ref{getour}.

We begin with the case $d=1$. After a translation and a dilation we can assume without loss of generality that $\Omega =(-1,1)$.

\begin{theorem}\label{weylsecond1d}
Let $\Omega=(-1,1)\subset\R$. Then
\begin{equation}
\label{eq:weylsecond1d}
E_n(H_\Omega^{(s)}) = \left( \frac{n\pi}{2} - \frac{(1-s)\pi}{4} \right)^{2s} + O(n^{-1})
\end{equation}
\end{theorem}

This theorem is due to Kwa\'snicki \cite{Kw2} and generalizes an earlier result \cite{KKMS} for $s=1/2$. A key role is played by a detailed analysis of the half line problem \cite{Kw}.

Asymptotics \eqref{eq:weylsecond1d} are remarkably precise. For $s>1/2$ they give the first three terms as $n\to\infty$. We also see that the $\liminf$ and the $\limsup$ of $N(\mu,H_\Omega^{(s)})-\pi^{-1}|\Omega|\mu^{1/(2s)}$ as $\mu\to\infty$ are finite, but do not coincide. (In fact, the $\limsup$ is positive for $s\in(0,1)$, which shows that the analogue of P\'olya's conjecture fails in the fractional case. This was first observed in \cite{KwLaSi}.)

\medskip

We now turn to the higher-dimensional case. The authors of \cite{BaKuSi} posed the problem to prove that, under suitable assumptions on $\Omega$, the quantity
$$
n^{-(2s-1)/d} \left( E_n(H_\Omega^{(s)}) - n^{2s/d} (2\pi)^{2s} \omega_d^{-2s/d} |\Omega|^{-2s/d} \right)
$$
has a limit. For $s=1$ this is a celebrated result by Ivrii \cite{Iv} which holds under the assumption that the set of periodic billiards has measure zero. In fact, after the first version of this review was submitted, Ivrii \cite{Iv2} announced a solution of the above problem for $s\in(0,1)$ under the same assumption.

The following theorem from \cite{FrGe} verifies the existence of a limit in the Ces\`aro sense, that is, the quantity
\begin{equation}
\label{eq:getoursumsecondlim}
N^{-(2s-1)/d} \left( N^{-1} \sum_{n=1}^N E_n(H_\Omega^{(s)}) - \frac{d}{d+2s} N^{2s/d} (2\pi)^{2s} \omega_d^{-2s/d} |\Omega|^{-2s/d} \right)
\end{equation}
has a limit. Just like \eqref{eq:getour} is equivalent to \eqref{eq:getouralt} and \eqref{eq:getoursum} is equivalent to \eqref{eq:getourgamma} with $\gamma=1$, the existence of the limit of \eqref{eq:getoursumsecondlim} is equivalent to the existence of the limit of 
\begin{equation}
\label{eq:getoursumaltsecondlim}
\mu^{-1-(d-1)/(2s)} \left( \tr (H_\Omega^{(s)}-\mu)_- -\mu^{1+d/(2s)} L_{1,d,s}^\cl\, |\Omega| \right).
\end{equation}
(These equivalences are elementary facts about sequences; see, e.g., \cite[Lemma 21]{FrGe}.) The advantage of \eqref{eq:getouralt}, \eqref{eq:getourgamma} and \eqref{eq:getoursumaltsecondlim} over \eqref{eq:getour}, \eqref{eq:getoursum} and \eqref{eq:getoursumsecondlim}, respectively, is that disjoint subsets of $\Omega$ have asymptotically an additive influence on the asymptotics, which allows for localization techniques.

The main result from \cite{FrGe} is

\begin{theorem}
\label{getoursumsecond}
For any $d\geq 1$ and $0<s<1$ there is a constant $B_{d,s}^\cl >0$ such that for any bounded domain $\Omega\subset\R^d$ with $C^1$ boundary,
\begin{equation}
\label{eq:getoursumsecond}
\lim_{\mu\to\infty} \mu^{-1-(d-1)/(2s)} \left( \tr (H_\Omega^{(s)} -\mu)_-  - \mu^{1+d/(2s)} L_{1,d,s}^\cl \, |\Omega| \right) = - B_{d,s}^\cl \, \sigma(\partial \Omega) \,.
\end{equation}
\end{theorem}

Here $\sigma(\partial\Omega)$ denotes the surface measure of $\partial\Omega$. In \cite{FrGe} this is stated for domains with $C^{1,\alpha}$ boundary, $0<\alpha\leq 1$, and a quantitative remainder whose order depends on $\alpha$. The same argument as in \cite{FrGe1}, however, yields the result for $C^1$ boundaries with a $o(1)$ remainder.

In \cite{FrGe} we obtain an expression for $B_{d,s}^\cl$ which is explicit enough to deduce that it is different (in fact, smaller) than the corresponding expression for the fractional power of the Dirichlet Laplacian. In order to state this precisely, we recall that there is a constant $\tilde B_{d,s}^\cl >0$ such that for any bounded domain $\Omega\subset\R^d$ with $C^1$ boundary,
\begin{equation*}
\lim_{\mu\to\infty} \mu^{-1-(d-1)/(2s)} \left( \tr ((-\Delta_\Omega)^s -\mu)_-  - \mu^{1+d/(2s)} L_{1,d,s}^\cl\, |\Omega| \right) = - \tilde B_{d,s}^\cl \, \sigma(\partial \Omega) \,;
\end{equation*}
see, e.g., \cite{FrGe0} for a proof for domains with $C^{1,\alpha}$ boundary, $0<\alpha\leq 1$, which again can be modified to yield the result for $C^1$ boundaries. We prove \cite[Sec. 6.4]{FrGe}
$$
B_{d,s}^\cl <\tilde B_{d,s}^\cl
$$
and deduce

\begin{corollary}\label{getoursumsecondcor}
For any bounded domain $\Omega\subset\R^d$ with $C^1$ boundary,
\begin{equation*}
\lim_{\mu\to\infty} \mu^{-1-(d-1)/(2s)} \left( \tr (H_\Omega^{(s)} -\mu)_-  - \tr ((-\Delta_\Omega)^s -\mu)_- \right) = - \!\left( B_{d,s}^\cl -\tilde B_{d,s}^\cl\right) \sigma(\partial \Omega) >0 \,.
\end{equation*}
\end{corollary}

Theorem \ref{getoursumsecond} implies via integration that
\begin{equation}
\label{eq:getourheatsecond}
\lim_{t\to 0} t^{(d-1)/(2s)} \left( \tr e^{-tH_\Omega^{(s)}} - t^{-d/(2s)} \frac{\omega_d\,\Gamma(1+\frac{d}{2s})}{(2\pi)^d} \,|\Omega| \right) = - \Gamma(2+\tfrac{d-1}{2s})\, B_{d,s}^\cl \, \sigma(\partial\Omega) \,.
\end{equation}
(This is essentially the argument that convergence in Ces\`aro sense implies convergence in Abel sense.) Asymptotics \eqref{eq:getourheatsecond} are, in fact, even true for $\Omega$ with Lipschitz boundary, as had earlier been shown in \cite{BaKuSi}. This extends the result from \cite{Br} for $s=1$ to the fractional case. See also \cite{BaKu08} for remainder terms in \eqref{eq:getourheatsecond} under stronger regularity assumptions on the boundary.

It seems to be unknown whether Theorem \ref{getoursumsecond} remains true for Lipschitz domains.

Asymptotics like \eqref{eq:getourheatsecond} have been shown for more general non-local operators, see, e.g., \cite{BaMiNa,PaSo,BoSi,Go}.


\subsection{Eigenvalue asymptotics for fractional Schr\"odinger operators}

The analogue of Theorem \ref{getour} for fractional Schr\"odinger operators is

\begin{theorem}\label{weyl}
Let $0<s<1$ and let $V$ be a continuous function on $\R^d$ with compact support. Then
\begin{equation}
\label{eq:weyl}
\lim_{\alpha\to\infty} \alpha^{-d/(2s)} N((-\Delta)^s+\alpha V) = (2\pi)^{-d} \omega_d \int_{\R^d} V_-^{d/(2s)}\,dx \,.
\end{equation}
\end{theorem}

Similarly to \eqref{eq:weyl} one can show that for any $\gamma>0$,
\begin{equation}
\label{eq:weylgamma}
\lim_{\alpha\to\infty} \alpha^{-\gamma-d/(2s)} \tr((-\Delta)^s+\alpha V)_-^\gamma = L_{\gamma,d,s}^\cl \int_{\R^d} V_-^{\gamma+d/(2s)}\,dx
\end{equation}
with $L_{\gamma,d,s}^\cl$ from \eqref{eq:weylconst}. The assumptions on $V$ for \eqref{eq:weyl} and \eqref{eq:weylgamma} to hold can be relaxed. In particular, for $d\geq 2$, as well as for $d=1$ and $0<s<1/2$, one can show that the asymptotics hold under the sole assumption $V_-\in L^{\gamma+d/(2s)}$. This will be explained after Theorem \ref{lt}. The case $d=1$ and $1/2\leq s<1$ is more subtle. In analogy with \cite{BiLa,NaSo} one might wonder whether there are $V\in L^{1/(2s)}$ for which $N((-\Delta)^s+\alpha V)$ grows faster than $\alpha^{1/(2s)}$ or like $\alpha^{1/(2s)}$ but with an asymptotic constant strictly larger than $\int_\R V_-^{1/(2s)}\,dx$. Apparently this question has not been studied.

We are also not aware of sharp remainder estimates or subleading terms in \eqref{eq:weyl} and \eqref{eq:weylgamma}.  Note that, due to the non-smoothness of $p\mapsto|p|^{2s}$ at $p=0$, the operator $(-h^2\Delta)^s+V$ is not an admissible operator in the sense of \cite{HeRo}. For a remainder bound for the massive analogue of \eqref{eq:weylgamma} with $\gamma=1/2$ we refer to \cite{SSS}.


\section{Bounds on sums of eigenvalues}\label{sec:unif}

\subsection{Berezin--Li--Yau inequalities}

In this subsection we discuss bounds on sums of eigenvalues of $H_\Omega^{(s)}$. The bounds in the following theorem are called Berezin--Li--Yau inequalities since they generalize the corresponding bounds for $s=1$ \cite{Be,LY} to the fractional case.

\begin{theorem}\label{bly}
Let $\Omega\subset\R^d$ be an open set of finite measure. Then for any $\mu>0$,
\begin{equation}
\label{eq:blyb}
\sum_n \left( E_n(H_\Omega^{(s)}) - \mu \right)_- \leq \mu^{1+d/(2s)} L_{1,d,s}^\cl\, |\Omega|
\end{equation}
and, equivalently, for any $N\in\N$,
\begin{equation}
\label{eq:blyly}
\sum_{n=1}^N E_n(H_\Omega^{(s)}) \geq \frac{d}{d+2s} (2\pi)^{2s} \omega_d^{-2s/d} |\Omega|^{-2s/d} N^{1+2s/d} \,.
\end{equation}
\end{theorem}

Inequality \eqref{eq:blyb} is a special case of a result in \cite{La}. To see that \eqref{eq:blyb} and \eqref{eq:blyly} are equivalent, denote the left and right side of \eqref{eq:blyb} by $f_l(\mu)$ and $f_r(\mu)$, respectively, by $g_l(\nu)$ the piecewise linear function which coincides with the left side of \eqref{eq:blyly} for $\nu=N\in\N$ and by $g_r(\nu)$ the right side of \eqref{eq:blyly} with $N$ replaced by a continuous variable $\nu$. Note that \eqref{eq:blyly} is equivalent to $g_l(\nu)\geq g_r(\nu)$ for \emph{all} $\nu>0$. We have defined four convex functions and we note that $f_\#$ and $g_\#$ are Legendre transforms of each other with $\#=l,r$. Thus, the equivalence follows from the fact that the Legendre transform reverses inequalities.

The important feature of \eqref{eq:blyb} and \eqref{eq:blyly} is that the constant on the right side coincides with the asymptotic value as $\mu$ or $N$ tend to infinity; see \eqref{eq:getoursum} and \eqref{eq:getourgamma}. For remainder terms in \eqref{eq:blyly} with the asymptotically correct power of $N$ we refer to \cite{YY}.

Bounding the left side of \eqref{eq:blyly} from above by $N E_N(H_\Omega^{(s)})$ or the left side of \eqref{eq:blyb} from below by $(\Lambda-\mu)_- N(\Lambda,H_\Omega^{(s)})$ and optimizing in $\Lambda<\mu$ we obtain
\begin{equation}
\label{eq:blynumber}
N(\Lambda,H_\Omega^{(s)}) \leq \left(\frac{d+2s}d\right)^{\frac d{2s}} \frac{\omega_d}{(2\pi)^d}\,|\Omega| \Lambda^{\frac d{2s}} \,,
\qquad
E_N(H_\Omega^{(s)}) \geq \frac{d}{d+2s} \frac{(2\pi)^{2s}}{\omega_d^{\frac{2s}d}} |\Omega|^{-\frac{2s}d} N^{\frac{2s}d} \,.
\end{equation}
It is a challenging open question (the fractional analogue of P\'olya's conjecture) whether the factors $((d+2s)/d)^{d/(2s)}$ and $d/(d+2s)$ can be removed in these bounds. It was recently shown \cite{KwLaSi} that P\'olya's conjecture fails in $d=1$ for $s\in(0,1)$ and in $d=2$ at least for all sufficiently small values of $s$. (As an aside, we mention that P\'olya's conjecture also fails for the Laplacian in two-dimensions with a constant magnetic field and that in this case the factors $((d+2s)/d)^{d/(2s)}=2$ and $d/(d+2s)=2$ are optimal \cite{FrLoWe}.)

\medskip

We finally mention a well known inequality for the heat kernel. From the maximum principle for the heat equation we know that the heat kernel $k_t(x,x')$ for $H_\Omega^{(s)}$ satisfies
$$
0 \leq k_t(x,x') \leq \int_{\R^d} e^{-t|p|^{2s}} e^{ip\cdot(x-x')} \frac{dp}{(2\pi)^d}
\qquad\text{for all}\ x,x'\in\Omega \,.
$$
(The right side is the heat kernel of $(-\Delta)^s$.) We evaluate this inequality for $x=x'$. If $H_\Omega^{(s)}$ has discrete spectrum (which is the case, for instance, if $|\Omega|<\infty$) and $\psi_n$ denote the normalized eigenfunctions corresponding to the $E_n(H_\Omega^{(s)})$, then we obtain
\begin{equation}
\label{eq:heatlocal}
\sum_n e^{-t E_n(H_\Omega^{(s)})} |\psi_n(x)|^2 \leq \frac{\omega_d\,\Gamma(1+d/(2s))}{(2\pi)^d} \, t^{-d/(2s)} 
\qquad\text{for all}\ x\in\Omega \,.
\end{equation}
By integration over $x\in\Omega$ we obtain
\begin{equation*}
\sum_n e^{-t E_n(H_\Omega^{(s)})} \leq \frac{\omega_d\,\Gamma(1+d/(2s))}{(2\pi)^d} \,|\Omega|\, t^{-d/(2s)} \,,
\end{equation*}
which, in turn, could have been obtained directly by integrating \eqref{eq:blyb} against $t^2 e^{-t\mu}$ over $\mu\in\R_+$. However, in some applications the local information in \eqref{eq:heatlocal} is crucial. For example, one useful consequence of \eqref{eq:heatlocal} comes by bounded the left side from below by $e^{-t\mu} \sum_{E_n(H_\Omega^{(s)})<\mu} |\psi_n(x)|^2$. Optimizing the resulting inequality over $t>0$ yields
\begin{equation}
\label{eq:locnumber}
\sum_{E_n(H^{(s)}_\Omega)<\mu} |\psi_n(x)|^2 \leq \frac{\omega_d\,\Gamma(1+d/(2s))}{(2\pi)^d} \left(\frac{2se}{d}\right)^{d/(2s)} \mu^{d/(2s)} \,.
\end{equation}
While yielding a worse constant than \eqref{eq:blynumber} when integrated over $x\in\Omega$, this a-priori bound on the `local number of eigenvalues' is useful when proving $\mu\to\infty$ asymptotics.


\subsection{Lieb--Thirring inequalities}

Lieb--Thirring inequalities \cite{LiTh} provide bounds of sums of powers of negative eigenvalues of Schr\"odinger operators in terms of integrals of the potential. They play an important role in the proof of stability of matter by Lieb and Thirring; see \cite{LiSe} for a textbook presentation. For further background and references about Lieb--Thirring inequalities we also refer to the reviews \cite{LaWe,Hu}.

The following theorem summarizes Lieb--Thirring inequalities for fractional Schr\"o\-dinger operators.

\begin{theorem}\label{lt}
Let $d\geq 1$, $0<s<1$ and
$$
\begin{cases}
\gamma \geq 1-1/(2s) & \text{if}\ d=1 \ \text{and}\ s> 1/2 \,, \\
\gamma >0 & \text{if}\ d=1 \ \text{and}\ s= 1/2 \,, \\
\gamma\geq 0 & \text{if}\ d\geq 2\ \text{or}\ d=1 \ \text{and}\ s<1/2 \,.
\end{cases}
$$
Then there is a $L_{\gamma,d,s}$ such that for all $V$,
\begin{equation}
\label{eq:lt}
\tr\left( (-\Delta)^s + V \right)_-^\gamma \leq L_{\gamma,d,s} \int_{\R^d} V_-^{\gamma+d/(2s)} \,dx \,.
\end{equation}
\end{theorem}

This theorem, with the additional assumption $\gamma>1-1/(2s)$ if $d=1$, $s>1/2$, appears in \cite{Da}, which also has explicit values for $L_{\gamma,d,s}$ in the physically most relevant cases. Since we have not found the case $\gamma=1-1/(2s)$ if $d=1$, $s>1/2$, in the literature, we provide a proof in Appendix \ref{app:ltcrit}.

To appreciate the strength of Theorem \ref{lt}, we note that by bounding the sum over all eigenvalues by a single one, we deduce from \eqref{eq:lt} that
$$
E_1((-\Delta)^s + V) \geq - \left( L_{\gamma,d,s} \int_{\R^d} V_-^{\gamma+d/(2s)} \,dx \right)^{1/\gamma} \,,
$$
which is the bound from Theorem \ref{ke} and which we have seen to be equivalent to the Sobolev inequality \eqref{eq:sobolevinterpol}. Moreover, replacing $V$ by $\alpha V$ and comparing with Theorem \ref{weyl} we see that the right side of \eqref{eq:lt} has the correct order of growth in the large coupling limit $\alpha\to\infty$. Thus, Theorem \ref{lt} shows that the semi-classical approximation is, up to a multiplicative constant, a uniform upper bound. This observation and a density argument based on Ky-Fan's eigenvalue inequality (see, e.g., \cite[Theorem 1.7]{Si2}) can be used to show that for $\gamma$ as in Theorem \ref{lt} the asymptotics \eqref{eq:weyl} and \eqref{eq:weylgamma} hold for all $V$ with $V_-\in L^{\gamma+d/(2s)}(\R^d)$.

Let us comment on the case $\gamma=0$ if $d=1$ and $s=1/2$. In this case it is easy to see that
$$
\inf_{\|\psi\|=1} \left( \|(-\Delta)^{1/4}\psi\|^2 + \int_\R V|\psi|^2\,dx \right) <0
\qquad\text{if}\ \int_\R V\,dx <0 \,,
$$
and so inequality \eqref{eq:lt} necessarily fails for $\gamma=0$. Remarkably, in this case one can show a \emph{reverse bound},
\begin{equation}
\label{eq:ltcritdim}
\tr\left( (-\Delta)^{1/2} + V \right)_-^0 \geq c \int_{\R} V_-\,dx
\qquad\text{if}\ V\leq 0 \,.
\end{equation}
(This is contained in \cite{Sh} up to a conformal transformation.)

While there has been substantial progress concerning the sharp constants in the $s=1$ analogue of Theorem \ref{lt}, no sharp constant seems to be known in the case $s<1$.

\medskip

Our final topic are Hardy--Lieb--Thirring inequalities. We recall \cite{He} that Hardy's inequality states that for $0<s<d/2$ and $\psi\in \dot H^s(\R^d)$, the homogeneous Sobolev space,
$$
\int_{\R^d} |p|^{2s} |\hat\psi(p)|^2 \,dp \geq \mathcal C_{s,d} \int_{\R^d} |x|^{-2s} |\psi(x)|^2\,dx
$$
with the sharp constant
$$
\mathcal C_{s,d} = 2^{2s}\, \frac{\Gamma((d+2s)/4)^2}{\Gamma((d-2s)/4)^2} \,.
$$
As a consequence, $(-\Delta)^s - \mathcal C_{s,d}|x|^{-2s}$ is a non-negative operator. The following theorem says that, up to avoiding the endpoint $\gamma=0$ and modifying the constant, the Lieb--Thirring inequalities from Theorem \ref{lt} remain valid when $(-\Delta)^s$ is replaced by $(-\Delta)^s - \mathcal C_{s,d}|x|^{-2s}$.

\begin{theorem}\label{hlt}
Let $d\geq 1$, $0<s<d/2$ and $\gamma>0$. Then there is a constant $L_{\gamma,d,s}^{\mathrm{HLT}}$ such that
\begin{equation}
\label{eq:hlt}
\tr\left( (-\Delta)^s - \mathcal C_{s,d}|x|^{-2s} + V \right)_-^\gamma \leq L_{\gamma,d,s}^{\mathrm{HLT}} \int_{\R^d} V_-^{\gamma+d/(2s)}\,dx \,.
\end{equation}
\end{theorem}

We emphasize that the assumption $s\leq 1$ is \emph{not} needed here. Moreover, arguing as before \eqref{eq:ltcritdim} one can show that the inequality does not hold for $\gamma=0$.

Theorem \ref{hlt} was initially proved for $s=1$ in \cite{EkFr} and then extended in \cite{FrLiSe} to $0<s<1$ (with $0<s<1/2$ if $d=1$). The full result is from \cite{Fr} and uses an idea from \cite{SSS}.

The proof in \cite{FrLiSe} (for $0<s\leq 1$) allows for the inclusion of a magnetic field. This leads to the proof of stability of relativistic matter with magnetic fields for nuclear charges up to and including the critical value; see also \cite{FrLiSe2}.

Let us briefly comment on the proof of Theorem \ref{hlt} in \cite{FrLiSe}, since this will also be relevant in the following. Similarly as after Theorem \ref{lt} we observe that by bounding the sum over all eigenvalues by a single one, we deduce from \eqref{eq:hlt} that
$$
E_1((-\Delta)^s - \mathcal C_{s,d}|x|^{-2s} + V) \geq - \left( L_{\gamma,d,s}^{\mathrm{HLT}} \int_{\R^d} V_-^{\gamma+d/(2s)} \,dx \right)^{1/\gamma} \,,
$$
which in turn, by the argument in the proof of Theorem \ref{ke}, is equivalent to the Hardy--Sobolev inequality
\begin{equation}
\label{eq:hltsingle}
\left( \|(-\Delta)^{s/2}\psi\|^2 - \mathcal C_{s,d} \||x|^{-s}\psi\|^2 \right)^\theta \|\psi\|^{2(1-\theta)} \geq C_{d,q,s} \|\psi\|_q^2
\end{equation}
with $1/(\gamma+d/(2s)) + 2/q = 1$ (and some $\theta\in(0,1)$ uniquely determined by scaling). The proof in \cite{FrLiSe} proceeds by first showing the latter inequality (at this point the assumption $s\leq 1$ enters through the use of the rearrangement inequality \eqref{eq:rearr} for $\|(-\Delta)^{s/2}\psi\|^2$) and then by proving, in an abstract set-up (see also \cite{FrLiSe3}), that a Sobolev inequality, in fact, \emph{implies} a Lieb--Thirring inequality. (To be more precise, there is an arbitrarily small loss in the exponent. For instance, \eqref{eq:hltsingle} for a given $q$ implies \eqref{eq:hlt} for any $\gamma$ with $1/(\gamma+d/(2s)) + 2/q < 1$. But since we want to prove \eqref{eq:hlt} for an open set of exponents $\gamma$, this loss is irrelevant for us.) This concludes our discussion of the proof of Theorem \ref{hlt}.

The Hardy inequalities discussed so far involve the function $|x|^{-2s}$ with a singularity at the origin. For convex domains there are also Hardy inequalities with the function $\dist(x,\Omega^c)^{-2s}$, or more generally, for arbitrary domains with the function
$$
m_{2s}(x) := \left( \frac{2\pi^{\frac{d-1}2}\Gamma(\frac{1+2s}{2})}{\Gamma(\frac{d+2s}{2})} \right)^{\frac1{2s}} \left( \int_{\Sph^{d-1}} \frac{d\omega}{d_\omega(x)^{2s}} \right)^{-\frac{1}{2s}} ,
$$
where $d_\omega(x):= \inf\{|t|:\ x+t\omega\not\in\Omega\}$. (We say `more generally' since one can show that $m_{2s}(x)\leq\dist(x,\Omega^c)$ for convex $\Omega$; see \cite{LoSl}.) The sharp Hardy inequality of Loss and Sloane \cite{LoSl} states that for $d\geq 2$, $1/2<s<1$, any open $\Omega\subset\R^d$ and any $\psi\in C_c^1(\Omega)$,
$$
\int_{\R^d} |p|^{2s}|\hat\psi(p)|^2\,dp \geq \mathcal C_{s}' \int_\Omega m_{2s}(x)^{-2s} |\psi(x)|^2\,dx
$$
with the sharp constant
$$
\mathcal C_{s}' = \frac{\Gamma(\frac{1+2s}{2})}{|\Gamma(-s)|} \, \frac{B(\frac{1+2s}{2},1-s)-2^{2s}}{2s \sqrt\pi} \,.
$$
This inequality is the fractional analogue of Davies' inequality \cite{D}. The fractional inequality in the special case of a half space is due to \cite{BoDy}.

The analogue of Theorem \ref{hlt} is

\begin{theorem}\label{hltbdry}
Let $d\geq 2$, $1/2<s<1$ and $\gamma\geq 0$. Then there is a constant $L_{\gamma,d,s}^{\mathrm{HLT'}}$ such that for all open $\Omega\subset\R^d$ and all $V$,
\begin{equation}
\label{eq:hltbdry}
\tr\left( H_\Omega^{(s)} - \mathcal C_{s}'m_{2s}^{-2s} + V \right)_-^\gamma \leq L_{\gamma,d,s}^{\mathrm{HLT'}} \int_{\Omega} V_-^{\gamma+d/(2s)}\,dx \,.
\end{equation}
\end{theorem}

We emphasize that, in contrast to Theorem \ref{hlt}, now $\gamma=0$ is allowed.

Theorem \ref{hltbdry} is the analogue of a result for $s=1$, $d\geq 3$ in \cite{FrLo}. Since it appears here for the first time, we comment briefly on its proof. Adapting an argument of Aizenman and Lieb \cite{AiLi} to our setting we see that it suffices to prove the inequality for $\gamma=0$. As in the proof of Theorem \ref{hlt} from \cite{FrLiSe} the first step is the `single function result', that is, the analogue of \eqref{eq:hltsingle}, which reads
\begin{equation}
\label{eq:hltbdrysingle}
\|(-\Delta)^{s/2}\psi\|^2 - \mathcal C_s' \|m_{2s}^{-s}\psi\|^2 \geq C_{d,s} \|\psi\|_{2d/(d-2s)}^2
\end{equation}
for $\psi\in C_c^1(\Omega)$. This inequality is proved in \cite{DyFr}. With \eqref{eq:hltbdrysingle} at hand one can apply the abstract machinery from \cite{FrLiSe3} in the same way as in \cite{FrLo} to obtain the theorem.


\section{Some further topics}\label{sec:omitted}

We conclude with some brief comments on further topics in the spectral theory of fractional Laplacians which are not included in the main part of this text.

(1) Positivity and uniqueness of the ground state. This is a classical result which can be derived using Perron--Frobenius arguments and the positivity of the heat kernel or by the maximum principle.

(2) Simplicity of excited states for radial fractional Schr\"odinger operators operators. This question has some relevance in non-linear problems and has recently been investigated in \cite{FrLe,FrLeSi} for Schr\"odinger operators with radially increasing potentials.

(3) Decay of eigenfunctions. In contrast to the local case $s=1$, the decay of eigenfunctions of Schr\"odinger operators with potentials tending to zero at infinity is only algebraic; see \cite{CaMaSi}. (Earlier bounds in the massive case are in \cite{Na,Na2}.) For bounds for growing potentials see, e.g., \cite{KaKu}.

(4) Shape of the ground state and of some excited states for the fractional Laplacian on a (convex) set. See \cite{BaKu04,BaKuMe} for some results in $d=1$ and \cite{Ku} for a related result in $d=2$. For superharmonicity in any $d$ for some $s$, see \cite{BadeBl}. For antisymmetry of the first excited state on a ball, see \cite{DyKuKw}. (This has also numerical methods for upper and lower bounds on the eigenvalues on a ball).

(5) Number of nodal domains. Is Sturm's bound in $d=1$ valid? Is Courant's bound in $d\geq 2$ valid? For some partial results, see \cite{BaKu04,FrLe,FrLeSi}.

(6) Regularity of eigenfunctions. Despite the non-locality of the fractional Laplacian, eigenfunctions of $(-\Delta)^s +V$ can be shown to be regular where $V$ is regular \cite{DFSS,DFSS2}. For improved H\"older continuity results for radial potentials, see \cite{Le}.

(7) Bounds on the gap $E_2(H_\Omega^{(s)})-E_1(H_\Omega^{(s)})$ for convex $\Omega$. See \cite{BaKu06,BaKu06b,Ka}; there are some conjectures in \cite{BaKu06}.

(8) Heat trace asymptotics for fractional Schr\"odinger operators and heat content asymptotics. See \cite{BaYo,Ac,Ac2,AcBa}.

(9) Many-body Coulomb systems. Stability of matter \cite{Da,LiTh2,LiYa1,FeLl,LiYa2,FrLiSe,FrLiSe2}. Proof of the Scott correction without \cite{FrSiWa,SSS} and with (self-generated) magnetic field \cite{EFS}.


\appendix

\section{Proof of \eqref{eq:formconst}}\label{app:const}

The following computation of $a_{s,d}$ is a slight simplification of \cite[Lemma 3.1]{FrLiSe}. It follows from Plancherel's theorem that
$$
\iint_{\R^d\times\R^d} \!\frac{|\psi(x)-\psi(y)|^2}{|x-y|^{d+2s}} \,dx\,dy = 
\iint_{\R^d\times\R^d}\!\frac{|\psi(x)-\psi(x+h)|^2}{|h|^{d+2s}} \,dx\,dh = 
\int_{\R^d} t(p) |\hat\psi(p)|^2\,dp
$$
with
$$
t(p) = \int_{\R^d} \frac{|1- e^{ip\cdot h}|^2}{|h|^{d+2s}} \,dh = 2 \int_{\R^d} \frac{1- \cos(p\cdot h)}{|h|^{d+2s}} \,dh \,.
$$
By homogeneity and rotation invariance we have
$$
t(p) = a_{d,s}^{-1}\, |p|^{2s}
$$
with
$$
a_{d,s}^{-1} = 2 \int_{\R^d} \frac{1- \cos(h_d)}{|h|^{d+2s}} \,dh \,.
$$
It remains to compute this integral. We begin with the case $d=1$, which is an exercise in complex analysis. First let $0<s<1/2$, so
$$
a_{1,s}^{-1} = 4 \re \int_0^\infty \frac{1- e^{ih}}{h^{1+2s}} \,dh \,.
$$
Since $(1-e^{iz})/z^{1+2s}$ is analytic in the upper right quadrant and sufficiently fast decaying as $|z|\to\infty$, we can move the integration from the positive real axis to the positive imaginary axis and obtain
$$
\int_0^\infty \frac{1- e^{ih}}{h^{1+2s}} \,dh = -i \int_0^\infty \frac{1- e^{-t}}{(it)^{1+2s}} \,dt = -i^{-2s} \int_0^\infty \frac{1- e^{-t}}{t^{1+2s}} \,dt \,.
$$
The integral here can be recognized as a gamma function. Indeed, we have if $\re z>0$,
$$
\Gamma(z) - \frac1z = -\int_0^1 \left(1-e^{-t}\right) t^{z-1}\,dt + \int_1^\infty e^{-t} t^{z-1}\,dt \,.
$$
Since the right side is analytic in $\{\re z>-1\}$, the formula extends to this region and, in particular,
$$
\Gamma(z) = - \int_0^\infty \left(1-e^{-t}\right) t^{z-1}\,dt
\qquad\text{if}\ -1<\re z<0 \,.
$$
Thus, we have shown that
$$
a_{1,s}^{-1} = 4 \re i^{-2s} \Gamma(-2s) = 4 \cos(\pi s) \Gamma(-2s) \,.
$$
Using the duplication formula $\Gamma(-2s)= \pi^{-1/2} 2^{2s-1} \Gamma(-s) \Gamma((1-2s)/2)$ and the reflection formula $\Gamma((1+2s)/2)\Gamma((1-2s)/2)= -\pi/\cos(\pi s)$ we obtain the claimed formula for $a_{1,s}$. When $1/2<s<1$, we start from
$$
a_{1,s}^{-1} = 4 \re \int_0^\infty \frac{1+ih- e^{ih}}{h^{1+2s}} \,dh
$$
and argue similarly using
$$
\Gamma(z) = - \int_0^\infty \left(1-t-e^{-t}\right) t^{z-1}\,dt
\qquad\text{if}\ -2<\re z<-1 \,.
$$
Finally, the formula for $s=1/2$ follows by continuity. This concludes the proof of \eqref{eq:formconst} for $d=1$.

Now let $d\geq 2$ and write $h=(h',h_d)\in \R^{d-1}\times\R$ and compute for fixed $h_d\in\R$
\begin{align*}
\int_{\R^{d-1}} \frac{dh'}{(h'^2+h_d^2)^{(d+2s)/2}} = \frac{b_{d,s}}{|h_d|^{1+2s}}
\end{align*}
with
$$
b_{d,s} = \int_{\R^{d-1}} \frac{d\eta}{(1+\eta^2)^{(d+2s)/2}} \,.
$$
Thus,
$$
a_{d,s}^{-1} = 2 b_{d,s} \int_\R \frac{1-\cos h_d}{|h_d|^{1+2s}} \,dh_d = b_{d,s} a_{1,s}^{-1} \,,
$$
and it remains to compute $b_{d,s}$. To do so we use \cite[(6.2.1), (6.2.2)]{AbSt} and obtain
\begin{align*}
b_{d,s} & = |\Sph^{d-2}| \int_0^\infty \frac{r^{d-2} dr}{(1+r^2)^{(d+2s)/2}}
= \frac{|\Sph^{d-2}|}{2} \int_0^\infty \frac{t^{(d-3)/2} dt}{(1+t)^{(d+2s)/2}} \\
& = \frac{|\Sph^{d-2}|}{2} \frac{\Gamma((d-1)/2)\,\Gamma((1+2s)/2)}{\Gamma((d+2s)/2)} 
= \pi^{(d-1)/2} \frac{\Gamma((1+2s)/2)}{\Gamma((d+2s)/2)} \,.
\end{align*}
This concludes the proof of \eqref{eq:formconst} for $d\geq 2$.


\section{Lieb--Thirring inequality in the critical case}\label{app:ltcrit}

Our goal in this appendix is to prove Theorem \ref{lt} in the critical case $d=1$, $1/2<s<1$ and $\gamma=1-1/(2s)$. Our argument will be a modification of Weidl's argument \cite{We} in the $s=1$ case (see also the unpublished manuscript \cite{Si}).

For $1/2<s<1$, any bounded interval $Q\subset\R$ and any $\psi\in H^s(Q)$, we define
$$
t_Q^{(s)}[\psi] := a_{1,s} \iint_{Q\times Q} \frac{|\psi(x)-\psi(y)|^2}{|x-y|^{1+2s}}\,dx\,dy \,,
$$
where $a_{1,s}$ is the constant from \eqref{eq:formconst}. We shall need the following Poincar\'e--Sobolev inequality for this quadratic form.

\begin{lemma}\label{sobolev}
Let $d=1$ and $1/2<s<1$. Then there is a constant $C_s$ such that for any bounded interval $Q\subset\R$ and any $\psi\in H^s(Q)$ with $\int_Q \psi\,dx =0$,
$$
\sup_Q |\psi|^2 \leq C_s |Q|^{2s-1} t_Q^{(s)}[\psi] \,.
$$
\end{lemma}

\begin{proof}
By a density argument we may assume that $\psi$ is continuous. We know from \cite{GRR} (with $\Psi(x)=x^2$ and $p(x) = |x|^{s+1/2}$) and a simple scaling argument that for any $a<b$ and any continuous function $\phi$ on $[a,b]$,
$$
\frac{|\phi(a)-\phi(b)|^2}{(b-a)^{2s-1}} \leq D_s \,\int_a^b \int_a^b \frac{|\phi(x)-\phi(y)|^2}{|x-y|^{1+2s}}\,dx\,dy
$$
with $D_s = (16(2s+1)/(2s-1))^2$. Since $\int_Q \psi\,dx =0$ there is a $c\in Q$ such that $\psi(c)=0$. Moreover, let $d\in Q$ be such that $|\psi(d)|=\sup|\psi|$. We apply the above inequality with $a=\min\{c,d\}$ and $b=\max\{c,d\}$ and note that $b-a\leq |Q|$ to obtain the lemma.
\end{proof}

The quadratic form $t_Q^{(s)}[\psi]$ is non-negative and closed in $L^2(Q)$ and therefore generates a self-adjoint operator, which we denote by $T_Q^{(s)}$. In some sense this corresponds to imposing Neumann boundary conditions on $\partial Q$.

\begin{lemma}\label{ltint}
Let $d=1$, $1/2<s<1$ and let $C_s$ be the constant from Lemma \ref{sobolev}. Let $Q\subset\R$ be a bounded interval and assume that $V\in L^1(Q)$ satisfies
$$
\alpha:= |Q|^{2s-1} \int_Q V_-\,dx < C_s^{-1} \,.
$$
Then $T_Q^{(s)}+V$ has at most one negative eigenvalue $E$ and this eigenvalue satisfies, if it exists,
$$
E \geq - \alpha^{-1/(2s-1)} (1-C_s\alpha)^{-1} \left( \int_Q V_-\,dx \right)^{2s/(2s-1)} \,.
$$
\end{lemma}

\begin{proof}
If $\psi\in H^s(Q)$ satisfies $\int_Q \psi\,dx =0$, then by Lemma \ref{sobolev}
\begin{align*}
t_Q^{(s)}[\psi] + \int_Q V|\psi|^2\,dx & \geq t_Q^{(s)}[\psi] - \int_Q V_-\,dx\, \sup_Q |\psi|^2 \geq  t_Q^{(s)}[\psi] \left( 1 - C_s |Q|^{2s-1}\int_Q V_-\,dx \right) \\
& =  t_Q^{(s)}[\psi] \left( 1 - C_s \alpha \right) \geq 0 \,.
\end{align*}
Thus, $E_2(T^{(s)}_Q+V)\geq 0$.

For general $\psi\in H^s(Q)$ we set $\psi_Q:= |Q|^{-1} \int_Q\psi\,dx$ and bound similarly, for any $\beta>0$,
\begin{align*}
t_Q^{(s)}[\psi] + \int_Q V|\psi|^2\,dx & \geq t_Q^{(s)}[\psi] - \int_Q V_-\,dx \left( \sup_Q |\psi-\psi_Q| + |\psi_Q| \right)^2 \\
& \geq t_Q^{(s)}[\psi] - \int_Q V_-\,dx \left( \left( C_s |Q|^{2s-1} t_Q^{(s)}[\psi] \right)^{1/2} + |Q|^{-1/2} \|\psi\| \right)^2 \\
& \geq t_Q^{(s)}[\psi] \left( 1- (1+\beta) C_s|Q|^{2s-1} \int_Q V_-\,dx \right) \\
& \qquad\qquad\ - (1+\beta^{-1})|Q|^{-1} \int_Q V_-\,dx\, \|\psi\|^2 \\
& = t_Q^{(s)}[\psi] \left( 1- (1+\beta) C_s \alpha \right) \\
& \qquad\qquad\ - (1+\beta^{-1}) \alpha^{-1/(2s-1)} \left( \int_Q V_-\,dx \right)^{2s/(2s-1)} \|\psi\|^2 \,.
\end{align*}
With the choice $\beta = (1-C_s\alpha)/(C_s\alpha)$ we finally obtain
$$
t_Q^{(s)}[\psi] + \int_Q V|\psi|^2\,dx \geq - \frac{1}{1-C_s\alpha} \alpha^{-1/(2s-1)} \left( \int_Q V_-\,dx \right)^{2s/(2s-1)} \|\psi\|^2 \,,
$$
which implies the lower bound on $E_1(T^{(s)}_Q+V)$ in the lemma.
\end{proof}

\begin{proof}[Proof of Theorem \ref{lt} for $d=1$, $1/2<s<1$, $\gamma=1-1/(2s)$]
Let $C_s$ be the constant from Lemma \ref{sobolev} and fix $0<\alpha<C_s$ to be chosen later. We claim that there are disjoint open intervals $Q_n$ whose closed union covers $\supp V_-$ and such that
$$
|Q_n|^{2s-1} \int_{Q_n} V_-\,dx = \alpha 
\qquad\text{for all}\ n \,.
$$
In fact, pick $x_0\in\R$ arbitrary and define $x_{k+1}$ inductively, given $x_k$, as follows: If $V_-\equiv 0$ on $(x_k,\infty)$ we stop the procedure. Otherwise, since $\ell\mapsto \ell^{2s-1} \int_{x_k}^{x_k+\ell} V_-\,dx$ is non-decreasing and unbounded, we can find $x_{k+1}>x_k$ such that 
$$
(x_{k+1}-x_k)^{2s-1} \int_{x_k}^{x_{k+1}} V_-\,dx = \alpha \,.
$$
Since $(x_{k+1}-x_k)^{2s-1} \geq \alpha/ \int_{x_0}^\infty V_-\,dx$, we will eventually cover $\supp V_-\cap [x_0,\infty)$. Now we repeat the same argument to the left of $x_0$. The $Q_n$'s are all the intervals $(x_k,x_{k+1})$.

We have
$$
\|(-\Delta)^{s/2}\psi\|^2 = a_{1,s} \iint_{\R\times\R} \frac{|\psi(x)-\psi(y)|^2}{|x-y|^{1+2s}}\,dx\,dy \geq 
\sum_n t_{Q_n}^{(s)}[\psi] \,,
$$
which, by the variational principle, implies that
$$
(-\Delta)^s + V \geq \sum_n \left( T_{Q_n}^{(s)} + V_{Q_n} \right),
$$
where $V_{Q_n}$ denotes the restriction of $V$ to $Q_n$, and therefore
$$
\tr\left( (-\Delta)^s + V \right)_-^{\frac{2s-1}{2s}} \leq \tr\left( \sum_n \left( T_{Q_n}^{(s)} + V_{Q_n} \right) \right)_-^{\frac{2s-1}{2s}} = \sum_n \tr \left( T_{Q_n}^{(s)} + V_{Q_n} \right)_-^{\frac{2s-1}{2s}}.
$$
According to Lemma \ref{ltint},
$$
\tr \left( T_{Q_n}^{(s)} + V_{Q_n} \right)_-^{\frac{2s-1}{2s}} \leq \alpha^{-\frac1{2s}} (1-C_s\alpha)^{-\frac{2s-1}{2s}} \int_{Q_n} V_-\,dx \,.
$$
Summing over $n$, we obtain
$$
\tr\left( (-\Delta)^s + V \right)_-^{\frac{2s-1}{2s}} \leq \alpha^{-\frac1{2s}} (1-C_s\alpha)^{-\frac{2s-1}{2s}} \int_{\R} V_-\,dx \,.
$$
We can optimize this in $\alpha$ by choosing $\alpha=1/(2sC_s)$ and obtain
$$
\tr\left( (-\Delta)^s + V \right)_-^{\frac{2s-1}{2s}} \leq C_s^\frac1{2s} \frac{2s}{(2s-1)^{\frac{2s-1}{2s}}} \int_{\R} V_-\,dx \,.
$$
This proves the theorem.
\end{proof}


\bibliographystyle{amsalpha}

\end{document}